\documentclass{ws-ijnt}

\begin{document}

\markboth{W. Ma}{A Generalized Digit Map}

\catchline{}{}{}{}{}
%
%%%%%%%%%%%%%%%%%%%%%%%%%%%%%%%%%%%%%%%%%%%%%%%%%%%%%%%%%%%%%%%%%%%%

\title{A Generalized Digit Map: Periodicity, Prouhet-Tarry-Escott Solutions, and Summation Identities}

\author{Wanli Ma}

\address{School of Mathematical Sciences, East China Normal University,\\
Shanghai, China\\
\email{mawanli271828314159@gmail.com} }

%\author{Second Author}

%\address{Group, Laboratory, Address\\
%City, State ZIP/Zone, Country\\
%author\_id@domain\_name }

\maketitle

\begin{history}
\received{(Day Month Year)}
\accepted{(Day Month Year)}
\end{history}

\begin{abstract}
	We investigate arithmetic properties of the sequence $b(n)=\mathcal{B}_{M,N}(n)\bmod M$ obtained from the base–$M$ to base–$N$ shift map $\mathcal{B}_{M,N}$.
	We prove that $b(n)$ is ultimately periodic exactly when every prime divisor of $M$ also divides $N$; in that case we bound (and, for prime powers, determine) the minimal period.
	When the condition fails, $b(n)$ supplies new solutions to the Prouhet–Tarry–Escott problem.
	To  analyze this situation we introduce a family of finite–difference identities and use them to evaluate two weighted multivariate polynomial sums, thereby extending identities that arise from the classical sum-of-digits function ($N=1$).
\end{abstract}

\keywords{sum-of-digits function; automatic sequence; Prouhet-Tarry-Escott(PTE) problem; periodicity of sequences.}

\ccode{Mathematics Subject Classification 2020: 11A63, 11P05, 11B83}

\section{Introduction}
	
	Fix integers \(M\ge 2\) and \(N\ge 1\).
	Writing an integer \(n\) in base \(M\) as
	\(
	n=\sum_{i=0}^{\infty} d_i(n)\,M^{\,i},
	d_i(n)\in\{0,\dots ,M-1\},
	\)
	we consider the map that keeps the same digit string but interprets it in base \(N\).

	\begin{definition}
		For \(M\ge 2\) and \(N\ge 1\) the \emph{base–shifting map}
		\begin{equation}\label{eq:base-shift}
			\mathcal{B}_{M,N}(n):=\sum_{i=0}^{\infty} d_i(n)\,N^{\,i},
		\end{equation}
		is well–defined for every \(n\in\mathbb N\).
	\end{definition}

	The sequence studied in this paper is the reduction of \(\mathcal{B}_{M,N}(n)\) modulo \(M\):
	\[
	b(n):=\mathcal{B}_{M,N}(n)\pmod{M},\qquad n\ge 0.
	\]
	This construction has appeared independently in recent literature. Shallit \cite{shallit2024rarefied} uses this map as a ``pseudopower'' function to analyze rarefied Thue-Morse sums, while Rampersad and Shallit \cite{rampersad2025rudin} employ the case $\mathcal{B}_{2,4}(n)$ in an automata-theoretic study of the Rudin-Shapiro sequence. While the combinatorial properties of $b(n)$ were explored in \cite{mawanli-2025-baseshift}, this paper focuses on its number-theoretic and algebraic structure.
	
	In the special case $N=1$, our map $\mathcal{B}_{M,1}(n)$ reduces to the classical base-$M$ \textit{sum-of-digits function}, $s_M(n)$. The sequence $s_M(n) \pmod M$ is a central object in combinatorics on words \cite{allouche2003automatic}. For example, its combinatorial properties, particularly the conditions for it to be overlap-free, were established by Allouche and Shallit \cite{AlloucheShallit2000sumofdigit} and later revisited with different methods by Cusick and Ciungu \cite{CusickCiungu2011sumofdigit}. Frid \cite{Frid-2001-overlap} also studied this sequence within the algebraic framework of symmetric D0L words.
	
	The generalized sequence $b(n)$ behaves differently based on the relationship between the prime factors of $M$ and $N$. The first part of our investigation provides a complete answer to the question of periodicity. We prove that the sequence $b(n)$ is ultimately periodic if and only if every prime factor of $M$ is also a prime factor of $N$ (Theorem~\ref{characterization of periodicity}).

	When this condition holds, we also provide an exact formula for the minimal period in cases where $M$ and $N$ are powers of the same prime, and establish general bounds for the period length otherwise.
	
	The non-periodic case, where $\operatorname{rad}(M) \nmid \operatorname{rad}(N)$, leads to an algebraic structure that we  analyze here. We first show that the sequence $b(n)$ provides a new family of solutions to the Prouhet-Tarry-Escott problem. This connection is established through a sequence of finite difference identities which generalize those known for the sum-of-digits function. The main results of this paper are the proofs of two closed-form evaluations for certain weighted multiple sums, which extend conjectures originally stated for the sum-of-digits function by Byszewski et al. \cite{byszewski-2015-identities} and later proved by Vignat and Wakhare \cite{vignat2018sumsupposition}. The first identity evaluates a weighted sum of powers of a linear multivariate polynomial, while the second evaluates a weighted sum of powers of a more complex polynomial involving both the variables $n_j$ and their images $\mathcal{B}_{M,N}(n_j)$.
	
	The paper is organized as follows. Section 2 is devoted to the periodic case, including the proof of the periodicity criterion and an analysis of the period length. Section 3 introduces the tools for the non-periodic case, including the connection to the Prouhet-Tarry-Escott problem and the definition of an auxiliary sequence $(\beta_k^{(p)})$. Section 4 establishes the key single and multiple finite difference identities. Finally, in Section 5, we apply these tools to prove the two main theorems that generalize the aforementioned conjectures.
	
\section{Periodicity of the Base-Shifting Sequence}
	
	\noindent The periodicity of digital sequences is a classical topic. A foundational result by Gelfond \cite{gelfond1968} concerns the sum-of-digits function, which in our notation is the case $N=1$. Gelfond's theorem implies that the sequence $\mathcal{B}_{M,1}(n) \pmod M$ is uniformly distributed for all $M \ge 2$, and is therefore never ultimately periodic. Our Theorem \ref{characterization of periodicity} complements this by providing a complete periodicity criterion for the entire family of sequences generated by $\mathcal{B}_{M,N}(n)$ for any $N \ge 1$, establishing the precise condition that separates these non-periodic cases from the periodic ones.
	\begin{theorem}\label{characterization of periodicity}
		Let $M \ge 2$ and $N \ge 1$ be integers. The sequence $b(n) = \mathcal{B}_{M,N}(n) \pmod M$ is ultimately periodic if and only if every prime factor of $M$ is also a prime factor of $N$.
	\end{theorem}
	This condition can be stated succinctly using the radical of an integer, $\operatorname{rad}(k)$, which is defined as the product of the distinct prime factors of $k$. The theorem could be stated as $b(n)$ is ultimately periodic if and only if $\operatorname{rad}(M) \mid \operatorname{rad}(N)$.
	
	\begin{proof}
		First, we prove sufficiency. Assume every prime factor of $M$ is also a prime factor of $N$. Let the prime factorization of $M$ be $M = p_1^{a_1} p_2^{a_2} \cdots p_r^{a_r}$. The condition implies that $N$ is divisible by $p_1 p_2 \cdots p_r$. Let $k_0 = \max(a_1, a_2, \dots, a_r)$. Then $p_i^{k_0} \mid N^{k_0}$ for each $i$, which implies $M \mid N^{k_0}$, so $N^{k_0} \equiv 0 \pmod M$.
		
		Consider any integer $n \ge 0$. We can write $n = q M^{k_0} + r$, where $q \ge 0$ and $0 \le r < M^{k_0}$. The base-$M$ representation of $n$ is formed by concatenating the digits of $q$ (shifted by $k_0$ positions) and the digits of $r$. Specifically, if $r = \sum_{i=0}^{k_0-1} d_i(r) M^i$ and $q = \sum_{j=0}^{\infty} d_j(q) M^j$, then $n = \sum_{i=0}^{k_0-1} d_i(r) M^i + \sum_{j=0}^{\infty} d_j(q) M^{j+k_0}$. Applying the base-shifting map $\mathcal{B}_{M,N}$ gives
		\begin{align*}
			\mathcal{B}_{M,N}(n) &= \sum_{i=0}^{k_0-1} d_i(r) N^i + \sum_{j=0}^{\infty} d_j(q) N^{j+k_0} \\
			&= \mathcal{B}_{M,N}(r) + N^{k_0} \mathcal{B}_{M,N}(q).
		\end{align*}
		Taking this modulo $M$, and since $N^{k_0} \equiv 0 \pmod M$, the second term vanishes, leaving $b(n) \equiv \mathcal{B}_{M,N}(r) \pmod M$. The value of $b(n)$ depends only on $r = n \pmod{M^{k_0}}$. This implies that for any $n_1, n_2 \ge 0$, if $n_1 \equiv n_2 \pmod{M^{k_0}}$, then $b(n_1) = b(n_2)$. This establishes that the sequence is periodic with a period $P$ that divides $M^{k_0}$.
		
		For necessity, we must show that if the sequence is ultimately periodic, the prime factor condition must hold. We use a proof by contradiction. Assume the sequence $b(n)_{n \ge 0}$ is ultimately periodic. Let its period be $P > 0$. This means there exists an integer $n_0$ such that for all $n \ge n_0$, we have $b(n+P) = b(n)$, which is equivalent to $\mathcal{B}_{M,N}(n+P) \equiv \mathcal{B}_{M,N}(n) \pmod M$. For the sake of contradiction, assume there exists a prime $p$ such that $p \mid M$ but $p \nmid N$.
		
		Choose an integer $k$ large enough such that $M^k > P$ and $M^k - P \ge n_0$. By the periodicity assumption, we must have $b(M^k) = b((M^k - P) + P) = b(M^k - P)$. We analyze each term modulo our chosen prime $p$.
		
		First, for the term $b(M^k)$, the base-$M$ expansion of $M^k$ is $d_k(M^k)=1$ and $d_i(M^k)=0$ for $i \ne k$. Thus, $\mathcal{B}_{M,N}(M^k) = 1 \cdot N^k = N^k$. This implies $b(M^k) \equiv N^k \pmod p$.
		
		Second, for the term $b(M^k - P)$, let the base-$M$ expansion of $P$ be $P = \sum_{i=0}^{j} p_i M^i$ with $p_j \ne 0$. We chose $k$ large enough so that $k > j$. The base-$M$ digits of $M^k-P$, denoted $d_i(M^k-P)$, are as follows: for $0 \le i \le j$, the digits are some values $d_i'$ which depend on $P$ but not on $k$; for $j+1 \le i \le k-1$, the digits are $M-1$; and for $i \ge k$, the digits are $0$. Applying the map $\mathcal{B}_{M,N}$ and reducing modulo $p$, we have
		\begin{align*}
			\mathcal{B}_{M,N}(M^k-P) &= \sum_{i=0}^{j} d_i' N^i + \sum_{i=j+1}^{k-1} (M-1) N^i \\
			&\equiv \sum_{i=0}^{j} d_i' N^i + \sum_{i=j+1}^{k-1} (-1) N^i \pmod p.
		\end{align*}
		Let $C = \sum_{i=0}^{j} d_i' N^i$. This value $C$ is a constant that depends on $P$ and $N$, but not on $k$. So, $b(M^k-P) \pmod p \equiv C - \sum_{i=j+1}^{k-1} N^i \pmod p$.
		
		Equating the two expressions from the periodicity condition yields the congruence 
		\begin{equation}
			N^k \equiv C - \sum_{i=j+1}^{k-1} N^i \pmod p,
			\end{equation}
			
		 which must hold for all sufficiently large integers $k$. We now analyze this based on the value of $N \pmod p$.
		
		Case A: $N \equiv 1 \pmod p$. The sum becomes a sum of ones, and the congruence becomes
		\begin{align*}
			1^k &\equiv C - \sum_{i=j+1}^{k-1} 1^i \pmod p \\
			1 &\equiv (C+j+1) - k \pmod p.
		\end{align*}
		The left side is constant, while the right side changes linearly with $k$. It is impossible for this to hold for all large $k$, which is a contradiction.
		
		Case B: $N \not\equiv 1 \pmod p$. The sum is a geometric series. Since $p \nmid N$ and $N \not\equiv 1 \pmod p$, its denominator $N-1$ is non-zero modulo $p$. The congruence becomes
		\begin{align*}
			N^k &\equiv C - \frac{N^k - N^{j+1}}{N-1} \pmod p \\
			N^k(N-1) &\equiv C(N-1) - (N^k - N^{j+1}) \pmod p \\
			N^{k+1} &\equiv C(N-1) + N^{j+1} \pmod p.
		\end{align*}
		Let $C' = C(N-1) + N^{j+1}$. This is another constant that does not depend on $k$. The condition is $N^{k+1} \equiv C' \pmod p$. This must hold for all sufficiently large $k$. However, since $p \nmid N$, the sequence $N^{k+1} \pmod p$ is periodic by Fermat's Little Theorem and is not constant (as we are in the case $N \not\equiv 1 \pmod p$). For the congruence to hold for all large $k$, we would need $N^{k+1} \equiv N^{k+2} \pmod p$, which implies $N \equiv 1 \pmod p$, a contradiction to the assumption of this case.
		
		In both cases, we reach a contradiction. Therefore, our initial assumption that "there exists a prime $p$ such that $p \mid M$ but $p \nmid N$" must be false. This completes the proof of necessity.
	\end{proof}
	
	\begin{theorem}\label{thm:period_prime_power}
		Let $p$ be a prime, and let $M = p^{a}$ and $N = p^{b}$ for integers $a,b \ge 1$. The sequence $b(n) = \mathcal{B}_{M,N}(n) \pmod M$ is purely periodic. Its minimal period is
		\[ P_{\min} = p^{e}, \quad \text{where} \quad e = at - b(t-1) \quad \text{and} \quad t = \left\lceil \frac{a}{b} \right\rceil. \]
	\end{theorem}

	\begin{proof}
		The proof consists of two steps: first, we establish that $P = p^e$ is a period of the sequence, and second, we prove that this period is minimal.
		
		Let $t = \lceil a/b \rceil$. By definition, this implies the inequality $b(t-1) < a \le bt$. A key consequence is that for any integer $k \ge t$, the $p$-adic valuation of $N^k$ satisfies $v_p(N^k) = bk \ge bt \ge a = v_p(M)$. This means $M \mid N^k$, which gives the crucial congruence $N^k \equiv 0 \pmod M$.
		
		Let $n \ge 0$ be any integer. We can decompose $n$ by separating its first $t-1$ base-$M$ digits from the rest, writing $n = r + q M^{t-1}$ where $0 \le r < M^{t-1}$ and $q \ge 0$. Let $q_0 = d_0(q)$ be the lowest digit of $q$. Using the linearity of $\mathcal{B}_{M,N}$ and the fact that higher powers of $N$ vanish modulo $M$, we find
		\begin{equation}
			\begin{split}
				\mathcal{B}_{M,N}(n) &= \mathcal{B}_{M,N}(r) + N^{t-1} \mathcal{B}_{M,N}(q) \\
				&= \mathcal{B}_{M,N}(r) + N^{t-1} \sum_{j=0}^{\infty} d_j(q)N^j \\
				&\equiv \mathcal{B}_{M,N}(r) + q_0 N^{t-1} \pmod M.
			\end{split}
		\end{equation}
		
		Now, consider the candidate period $P = p^e$. Let $s = a - b(t-1)$. From the inequality for $t$, we have $0 < s \le b$. We can express the period as $P = p^s M^{t-1}$. Then for any $n$, we have $n+P = r + (q+p^s)M^{t-1}$. Applying the same logic to $n+P$ yields $\mathcal{B}_{M,N}(n+P) \equiv \mathcal{B}_{M,N}(r) + d_0(q+p^s) N^{t-1} \pmod M$. Since $d_0(q+p^s) \equiv q+p^s \pmod M \equiv q_0 + p^s \pmod M$, the difference is
		\begin{align*}
			\mathcal{B}_{M,N}(n+P) - \mathcal{B}_{M,N}(n) &\equiv \left( d_0(q+p^s) - q_0 \right) N^{t-1} \pmod M \\
			&\equiv p^s N^{t-1} \pmod M.
		\end{align*}
		This final term evaluates to $p^s N^{t-1} = p^{a-b(t-1)} (p^b)^{t-1} = p^a = M$. Thus, $\mathcal{B}_{M,N}(n+P) - \mathcal{B}_{M,N}(n) \equiv 0 \pmod M$. This holds for all $n \ge 0$, so the sequence is purely periodic with period $P=p^e$.
		
		To prove minimality, we show that no smaller power of $p$ is a period. It suffices to show that $P' = p^{e-1}$ is not a period. If it were, we would require $b(P') = b(0) = 0$. We evaluate $b(P')$. We can write $P' = p^{s-1}M^{t-1}$. To see that $p^{s-1}$ is a valid single digit in base $M$, we must check that $p^{s-1} < M = p^a$, which is equivalent to $s-1 < a$. This holds because $s-1-a = (a-b(t-1))-1-a = -b(t-1)-1 < 0$ for $b,t \ge 1$. Thus, the base-$M$ representation of $P'$ has only one non-zero digit, $d_{t-1}(P') = p^{s-1}$. Applying the map gives
		\[ \mathcal{B}_{M,N}(P') = d_{t-1}(P') N^{t-1} = p^{s-1} N^{t-1} = p^{s-1} (p^b)^{t-1} = p^{s-1+b(t-1)} = p^{a-1}. \]
		So, $b(P') \equiv p^{a-1} \pmod M$. Since $a \ge 1$, $p^{a-1}$ is not divisible by $M=p^a$, so $b(P') \not\equiv 0 \pmod M$. We have found $b(P') \ne b(0)$, so $P'$ is not a period. From the sufficiency proof of Theorem \ref{characterization of periodicity}, we know a power of $p$ is a period, which implies the minimal period must also be a power of $p$. Since $p^e$ is a period and $p^{e-1}$ is not, we conclude that the minimal period is $P_{\min} = p^e$.
	\end{proof}

	\begin{proposition}\label{Estimation of the Minimal Period}[Estimation of the Minimal Period]
		Let $M = \prod_{p|M} p^{a_p}$ and $N = K \cdot \prod_{p|M} p^{b_p}$ be integers with $a_p, b_p \ge 1$ and $\gcd(K,M)=1$. The sequence $b(n) = \mathcal{B}_{M,N}(n) \pmod M$ is ultimately periodic. Define $\overline{t} := \max_{p|M} \{ \lceil a_p / b_p \rceil \}$. Then the minimal period, $P_{\min}$, is bounded by
		\[ M \le P_{\min} \le M^{\overline{t}}. \]
	\end{proposition}
	
	\begin{proof}
		We first establish the upper bound by showing that $M^{\overline{t}}$ is a period of the sequence $b(n)$. By the definition of $\overline{t}$, for any prime $q$ that divides $M$, we have $\overline{t} \ge \lceil a_q / b_q \rceil$, which implies $b_q \overline{t} \ge a_q$. In terms of $p$-adic valuations, this means $v_q(N^{\overline{t}}) = b_q \overline{t} \ge a_q = v_q(M)$. Since this holds for all prime factors $q$ of $M$, it follows that $M \mid N^{\overline{t}}$, and so $N^{\overline{t}} \equiv 0 \pmod M$.
		
		Now, let $n \ge 0$ be an arbitrary integer. By the division algorithm, we can write $n = q \cdot M^{\overline{t}} + r$ for some integers $q \ge 0$ and $0 \le r < M^{\overline{t}}$. The base-$M$ digits of $n$ are effectively a concatenation of the digits of $r$ and the shifted digits of $q$. The map $\mathcal{B}_{M,N}$ then gives
		\begin{align*}
			\mathcal{B}_{M,N}(n) &= \mathcal{B}_{M,N}(r) + N^{\overline{t}} \mathcal{B}_{M,N}(q) \\
			&\equiv \mathcal{B}_{M,N}(r) \pmod M,
		\end{align*}
		where the second term vanishes due to the congruence $N^{\overline{t}} \equiv 0 \pmod M$. This shows that the value of $b(n)$ depends only on $r = n \pmod{M^{\overline{t}}}$. Consequently, the sequence is periodic with a period that divides $M^{\overline{t}}$. Since the minimal period $P_{\min}$ must divide any period, we have $P_{\min} \le M^{\overline{t}}$, establishing the upper bound.
		
		Next, we establish the lower bound, $P_{\min} \ge M$. Let $P$ be any period of the sequence. Assume for contradiction that $0 < P < M$. Since the sequence is ultimately periodic, there exists an integer $n_0$ such that $b(n+P) = b(n)$ for all $n \ge n_0$. To derive a contradiction, we choose an integer $k$ large enough such that $n = kM \ge n_0$.
		
		For this choice of $n$, the base-$M$ representation of $n=kM$ has its lowest digit $d_0(kM)=0$, and its higher digits are those of $k$ shifted, so $\mathcal{B}_{M,N}(kM) = N \cdot \mathcal{B}_{M,N}(k)$. Because $0<P<M$, the number $n+P=kM+P$ has a lowest digit $d_0(kM+P)=P$, while its higher digits are identical to those of $kM$. This gives
		\begin{align*}
			\mathcal{B}_{M,N}(kM+P) &= d_0(kM+P)N^0 + \sum_{i=1}^{\infty} d_i(kM)N^i \\
			&= P + N \cdot \mathcal{B}_{M,N}(k).
		\end{align*}
		The periodicity condition $b(kM+P) = b(kM)$ implies that these two expressions must be congruent modulo $M$:
		\[ P + N \cdot \mathcal{B}_{M,N}(k) \equiv N \cdot \mathcal{B}_{M,N}(k) \pmod M. \]
		Subtracting the common term from both sides yields $P \equiv 0 \pmod M$. This contradicts our assumption that $0 < P < M$. Therefore, no period can exist in this range. Any positive period $P$ must satisfy $P \ge M$, and this must also hold for the minimal period. This completes the proof.
	\end{proof}
	
	Numerical evidence suggests the following more precise formula for the minimal period under certain conditions, which we state as a conjecture.
	\begin{conjecture}\label{conjecture about general period}
		Let $M = \prod_{p|M} p^{a_p}$ and $N = K \cdot \prod_{p|M} p^{b_p}$ be integers with $a_p, b_p \ge 1$ and $\gcd(K,M) = 1$. If for every prime factor $p|M$, we have $t_p := \lceil a_p / b_p \rceil \le 2$, then the minimal period of $b(n)$ is given by
		\[ P_{\min} = \prod_{p|M} p^{e_p}, \]
		where
		\[ e_p = a_p t_p - b_p (t_p - 1). \]
	\end{conjecture}
	
\section{Non-periodic case: Prouhet–Tarry–Escott Connections}

	\noindent Previously, we examined the case where $\operatorname{rad}(M) \mid \operatorname{rad}(N)$, which leads to ultimately periodic sequences. We now turn to the non-periodic case where $\operatorname{rad}(M) \nmid \operatorname{rad}(N)$. The resulting sequences exhibit properties analogous to those of the classical sum-of-digits function, which we will explore in this section.

	\begin{theorem}\label{zeros of G_p(z) at z=1}
		Let $M \ge 2, N \ge 1, p \ge 0$ be integers such that $\operatorname{rad}(M) \nmid \operatorname{rad}(N)$. Let $\xi$ be a primitive $M$-th root of unity. Define the polynomial $G_p(z)$ as
		\begin{equation} \label{defi:G_p(z)}
			G_p(z) := \sum_{m=0}^{M^{p+1}-1} \xi^{\mathcal{B}_{M,N}(m)} z^m.
			 \end{equation}
		Then $z=1$ is a root of $G_p(z)$ with multiplicity at least $p+1$. That is, $G_p^{(j)}(1) = 0$ for all integers $j$ in the range $0 \le j \le p$.
	\end{theorem}

	\begin{proof}
		The $j$-th derivative of $G_p(z)$ at $z=1$ is a sum weighted by the falling factorial polynomial $m(m-1)\cdots(m-j+1)$. Since the standard monomials $\{m^k\}_{k=0}^j$ form a basis for the space of polynomials of degree at most $j$, showing that $G_p^{(j)}(1)=0$ for $0 \le j \le p$ is equivalent to showing that certain weighted power sums vanish. Specifically, we will prove that for any integer $k$ such that $0 \le k \le p$, the following sum is zero:
		\[ S_k := \sum_{m=0}^{M^{p+1}-1} m^k \xi^{\mathcal{B}_{M,N}(m)} = 0. \]
		The summation range consists of integers $m$ with at most $p+1$ digits in base $M$. We can express each such $m$ as $m = \sum_{i=0}^p d_i M^i$, where $d_i \in \{0, \dots, M-1\}$. Consequently, $\mathcal{B}_{M,N}(m) = \sum_{i=0}^p d_i N^i$. We can rewrite the sum $S_k$ by summing over all possible digit vectors $(d_0, \dots, d_p)$:
		\[ S_k = \sum_{d_0=0}^{M-1} \cdots \sum_{d_p=0}^{M-1} \left(\sum_{i=0}^p d_i M^i\right)^k \xi^{\sum_{i=0}^p d_i N^i}. \]
		Using the multinomial theorem to expand $(\sum d_i M^i)^k$ and rearranging the order of summation gives
		\begin{equation}
			\begin{split}
				S_k &= \sum_{d_0, \dots, d_p} \left( \sum_{k_0+\dots+k_p=k} \frac{k!}{k_0! \cdots k_p!} \prod_{i=0}^p (d_i M^i)^{k_i} \right) \left( \prod_{i=0}^p \xi^{d_i N^i} \right) \\
				&= \sum_{k_0+\dots+k_p=k} \frac{k!}{k_0! \cdots k_p!} \left(\prod_{i=0}^p M^{ik_i}\right) \left( \sum_{d_0, \dots, d_p} \prod_{i=0}^p d_i^{k_i} \xi^{d_i N^i} \right).
			\end{split}
		\end{equation}
		The innermost sum over the digits separates into a product of sums over each digit:
		\begin{equation}
			 \sum_{d_0, \dots, d_p} \prod_{i=0}^p d_i^{k_i} \xi^{d_i N^i} = \prod_{i=0}^p \left( \sum_{d=0}^{M-1} d^{k_i} (\xi^{N^i})^{d} \right).
			\end{equation}
		Consider any term in the multinomial expansion, which corresponds to a tuple of non-negative integers $(k_0, \dots, k_p)$ such that $\sum_{i=0}^p k_i = k$. Since there are $p+1$ such integers and their sum $k$ is at most $p$, at least one of these integers must be zero. Let $i_0$ be an index for which $k_{i_0}=0$. The corresponding factor in the product of sums becomes
		\[ \sum_{d=0}^{M-1} d^{0} (\xi^{N^{i_0}})^{d} = \sum_{d=0}^{M-1} (\xi^{N^{i_0}})^{d}. \]
		This is a finite geometric series. The condition $\operatorname{rad}(M) \nmid \operatorname{rad}(N)$ ensures there is a prime $q$ dividing $M$ but not $N$. Since $\xi$ is a primitive $M$-th root of unity, $\xi^x=1$ if and only if $M \mid x$. As $q \mid M$ but $q \nmid N^{i_0}$, we know $M$ cannot divide $N^{i_0}$, so $\xi^{N^{i_0}} \ne 1$. The sum of the geometric series is therefore
		\[ \frac{(\xi^{N^{i_0}})^M - 1}{\xi^{N^{i_0}} - 1} = \frac{(\xi^M)^{N^{i_0}} - 1}{\xi^{N^{i_0}} - 1} = \frac{1^{N^{i_0}} - 1}{\xi^{N^{i_0}} - 1} = 0. \]
		Since every term in the multinomial expansion of $S_k$ contains at least one such zero factor, the entire sum is zero. Thus, $S_k = 0$ for all $0 \le k \le p$. As this implies $G_p^{(j)}(1)=0$ for $0 \le j \le p$, the proof is complete.
	\end{proof}
	The Prouhet-Tarry-Escott (PTE) problem \cite{bolker-2016-PTE-problem,raghavendran-2019-review-PTE-problem,nguyen-2016-proof-PTE-problem} is a classic problem in additive number theory that seeks two (or several) distinct multisets of integers, each with $n$ elements, such that the sum of their $k$-th powers are equal for $k=1, \ldots, d$. The following theorem provides solutions to some specific type of PTE problems.
	\begin{theorem}
		Let $M \ge 2$ and $p \ge 1$ be integers. For an integer $N$ in the range $\{1, 2, \dots, M-1\}$, consider the partition of the set $I=\{0,1,\dots , M^{p+1}-1\}$ into $M$ subsets based on the base-shifting map:
		$$
		T_j(N) := \bigl\{\,m\in I\big|  \mathcal{B}_{M,N}(m)\equiv j\pmod{M}\bigr\}, \qquad j=0,1,\dots ,M-1.
		$$
		This construction yields a solution to the Prouhet-Tarry-Escott problem of degree $p$—meaning the sum of the $k$-th powers of the elements is the same across all sets $T_j$ for $k=0, 1, \dots, p$—if and only if $\gcd(M,N)=1$.
		
		Furthermore, distinct values of $N$ satisfying this condition produce distinct partitions. Consequently, the total number of unique solutions generated by this method is $\phi(M)$, the value of Euler's totient function, which counts the integers $N \in \{1, \dots, M-1\}$ for which $\gcd(M,N)=1$.
	\end{theorem}
	
	\begin{proof}
		The proof is established in three parts. First, we show the condition on $N$ is sufficient for the partition to be a PTE solution. Second, we show it is necessary. Finally, we demonstrate that distinct choices of $N$ yield unique partitions.
		
		Let $\xi = e^{2\pi i/M}$ be a primitive $M$-th root of unity.
		First, we prove sufficiency. Assume $\gcd(M,N)=1$.  For $k=0,\dots ,p$ and $l=1,\dots ,M-1$ set
		\[
		S_k(l):=\sum_{m=0}^{M^{p+1}-1} m^{\,k}\,(\xi^{\,l})^{\mathcal{B}_{M,N}(m)} .
		\]
		As in Theorem 3.1 we expand $m$ digit-wise; every term of the resulting multinomial contains a factor
		\[
		\sum_{d=0}^{M-1} (\xi^{\,lN^{\,i}})^d
		=\frac{(\xi^{\,lN^{\,i}})^M-1}{\xi^{\,lN^{\,i}}-1},
		\]
		for some index $i\in\{0,\dots ,p\}$.  Because $\gcd(M,N)=1$, $N^{\,i}$ is coprime to $M$,
		hence $\xi^{\,lN^{\,i}}\neq1$; the numerator equals $1-1=0$, so the factor
		vanishes.  Consequently $S_k(l)=0$ for all stated $k$ and $l$.
		Define $C_{j,k}:=\sum_{m\in T_j(N)} m^{\,k}$.  Grouping the summands gives
		\[
		S_k(l)=\sum_{j=0}^{M-1}C_{j,k}(\xi^{\,l})^{j}=0\quad(l=1,\dots ,M-1).
		\]
		Let $P(z):=\sum_{j=0}^{M-1}C_{j,k}z^{\,j}$.  The previous equation says
		$P(\xi^{\,l})=0$ for $l=1,\dots ,M-1$; thus the $M-1$ distinct numbers
		$\xi,\xi^{\,2},\dots ,\xi^{\,M-1}$ are roots of the degree-$\le M-1$ polynomial $P$.
		Hence
		\[
		P(z)=K\,(1+z+\dots +z^{\,M-1})
		\]
		for some constant $K$, so $C_{0,k}=C_{1,k}=\dots =C_{M-1,k}=K$.
		Using also $S_k(0)=\sum_{j}C_{j,k}=C_k^{\text{tot}}$ we get
		$K=C_k^{\text{tot}}/M$; therefore all $T_j(N)$ have identical power sums
		up to degree $p$, i.e. they solve the Prouhet–Tarry–Escott problem.
		
		Next we prove necessity:Suppose, to the contrary, that \(\gcd(M,N)=d>1\).
		Pick a prime \(q\mid d\).  Write
		\[
		l:=\frac{M}{q}\quad(1\le l\le M-1),\qquad
		\xi:=e^{2\pi i/M},\qquad
		\eta:=\xi^{\,l}=e^{2\pi i\,l/M}.
		\]
		Because \(q\mid N\) we have \(N= qN_0\) for some integer \(N_0\); hence
		\begin{equation}\label{power of eta equals one}
			\eta^{\,N}=\xi^{\,lN}= \xi^{\,M N_0}=1,\qquad
			\eta^{\,N^{\,i}}=1\quad(i\ge 1).
		\end{equation}
		
		Consider \(S_{1}(l)=\displaystyle\sum_{m=0}^{M^{p+1}-1} m\,\eta^{\mathcal{B}_{M,N}(m)}\).
		Split each integer \(m\) into its least base-\(M\) digit and the rest:
		write \(m=r+M\,t\) with
		
		\[
		0\le r<M,\quad 0\le t<M^{p}-1.
		\]
		
		Because of (\ref{power of eta equals one}) we have
		
		\[
		\eta^{\mathcal{B}_{M,N}(m)}
		 = \eta^{\,\mathcal{B}_{M,N}(r)}\,
		\eta^{\,N\,\mathcal{B}_{M,N}(t)} = \eta^{\,\mathcal{B}_{M,N}(r)},
		\]
		since \(N\,\mathcal{B}_{M,N}(t)\) is a multiple of \(N\) and hence of \(q\).
		Therefore
		
		\begin{equation}\label{expansion of S(1,l)}
			\begin{aligned}
				S_{1}(l)
				&= \sum_{t=0}^{M^{p}-1}\sum_{r=0}^{M-1} (Mt+r)\,\eta^{\,\mathcal{B}_{M,N}(r)}\\[4pt]
				&= M\Bigl(\sum_{t=0}^{M^{p}-1} t\Bigr)\Bigl(\sum_{r=0}^{M-1}\eta^{\,\mathcal {B}_{M,N}(r)}\Bigr)
				 + 
				M^{p}\sum_{r=0}^{M-1} r\,\eta^{\,\mathcal{B}_{M,N}(r)}.      
			\end{aligned}
		\end{equation}
		
		Now \(\mathcal{B}_{M,N}(r)=r\) for \(0\le r<M\), so the inner geometric sums simplify:
		
		\[
		\sum_{r=0}^{M-1}\eta^{\,r}=\frac{\eta^{M}-1}{\eta-1}=0
		\quad(\eta\neq1),
		\]
		while
		
		\[
		\sum_{r=0}^{M-1} r\,\eta^{\,r}
		=\eta\frac{d}{d\eta}\Bigl(\frac{\eta^{M}-1}{\eta-1}\Bigr)
		=\frac{M}{\eta^{-1}-1}\neq0.
		\]
		
		Equation (\ref{expansion of S(1,l)}) therefore yields
		
		\begin{equation}\label{S(1,l) is not zero}
			S_{1}(l)=M^{p}\,\frac{M}{\eta^{-1}-1} \neq 0.  
		\end{equation}
		
		On the other hand, if the blocks \(T_j(N)\) formed a PTE solution of degree
		\(p\ge1\) we would have
		
		\[
		C_{0,1}=C_{1,1}=\dots=C_{M-1,1}=:C_{1}.
		\]
		
		Consequently
		
		\[
		S_{1}(l)=\sum_{j=0}^{M-1}C_{j,1}\,\xi^{\,lj}
		=C_{1}\sum_{j=0}^{M-1}\xi^{\,lj}=C_{1}\cdot0=0,
		\]
		
		contradicting (\ref{S(1,l) is not zero}).  Hence the assumed PTE property is impossible when
		\(\gcd(M,N)>1\).
		Necessity is proved.
		
		Uniqueness of the partitions for coprime \(N\):
		
		Assume \(\gcd(M,N_1)=\gcd(M,N_2)=1\) and \(N_1\not\equiv N_2\pmod M\).
		Take \(m=1\).  Then
		
		\[
		\mathcal B_{M,N_i}(1)=1\quad(i=1,2).
		\]
		
		Therefore
		
		\[
		1\in T_{1\bmod M}(N_i).
		\]
		
		Because \(N_1\not\equiv N_2\pmod M\), these residues are different,
		so the sets \(T_j(N_1)\) and \(T_j(N_2)\) cannot coincide as partitions
		of \(\{0,\dots ,M^{p+1}-1\}\).
		Hence distinct coprime residues \(N\) yield distinct block systems.
		Since exactly \(\varphi(M)\) numbers in \(\{1,\dots ,M-1\}\) are coprime
		to \(M\), the construction produces precisely \(\varphi(M)\) different
		PTE solutions.
	\end{proof}

	\begin{definition}
		Let $M \ge 2, N \ge 1, p \ge 0$ be integers such that $\operatorname{rad}(M) \nmid \operatorname{rad}(N)$, and let $\xi$ be a primitive $M$-th root of unity. We define a sequence $(\beta_k^{(p)})$ via its generating function $F_p(z)$:
		\begin{equation}\label{defi:F_p(z)}
			 F_p(z) := \frac{G_p(z)}{(1-z)^{p+1}} = \sum_{k \ge 0}\beta_k^{(p)}z^k,
			  \end{equation}
			  
		where $G_p(z)$ is the polynomial $\sum_{m=0}^{M^{p+1}-1} \xi^{\mathcal{B}_{M,N}(m)} z^m$ from Theorem \ref{zeros of G_p(z) at z=1}.
	\end{definition}
	
	\begin{remark}
		In Theorem \ref{zeros of G_p(z) at z=1}, we proved that $G_p(z)$ has a root of multiplicity at least $p+1$ at $z=1$. This ensures that the factor $(1-z)^{p+1}$ in the numerator and denominator cancels, confirming that $F_p(z)$ is indeed a polynomial. Its degree is at most $M^{p+1}-1 - (p+1) = M^{p+1}-p-2$, so the sequence $(\beta_k^{(p)})$ has finite support. This makes the definition sound.
	\end{remark}
	
	The following theorem provides two fundamental identities that relate the sequences at different levels of $p$.
	
	\begin{theorem}\label{thm:convolution_identities}
		Let the sequence $(\beta_k^{(p)})$ be defined as above. The first identity expresses $\beta_n^{(p+1)}$ as a sum involving earlier terms of the original sequence: for $p \ge 0$ and $n \ge 0$,
		\begin{equation}\label{first_convolution_identities}
			 \beta_n^{(p+1)} = \sum_{k=0}^{n} \xi^{\mathcal{B}_{M,N}(k)} \binom{n-k+p+1}{p+1}, 
			\end{equation}
			
		where the sum implicitly runs over $k < M^{p+2}$. 
		The second identity provides an inversion formula, expressing the original sequence in terms of $(\beta_k^{(p)})$: for $p \ge 0$ and $0 \le n < M^{p+1}$,
		\begin{equation}\label{second_convolution_identities}
			 \xi^{\mathcal{B}_{M,N}(n)} = \sum_{k=0}^{n} (-1)^k \binom{p+1}{k} \beta_{n-k}^{(p)}. \end{equation}
	\end{theorem}
	
	\begin{proof}
		Both results follow directly from analyzing the coefficients of the defining relation $G_p(z) = (1-z)^{p+1} F_p(z)$.
		
		To prove the first identity, we find the coefficient of $z^n$ in the expression $F_{p+1}(z) = G_{p+1}(z) \cdot (1-z)^{-(p+2)}$. This coefficient, $\beta_n^{(p+1)}$, is given by the Cauchy product of the coefficients of the two series involved. The first series is $G_{p+1}(z) = \sum_{k=0}^{M^{p+2}-1} \xi^{\mathcal{B}_{M,N}(k)} z^k$. The second is the generalized binomial series $(1-z)^{-(p+2)} = \sum_{j=0}^{\infty} \binom{j+p+1}{p+1} z^j$. The convolution formula for the coefficient of $z^n$ is
		\[ \beta_n^{(p+1)} = \sum_{k=0}^{n} \left( \text{coeff of } z^k \text{ in } G_{p+1}(z) \right) \cdot \left( \text{coeff of } z^{n-k} \text{ in } (1-z)^{-(p+2)} \right), \]
		which directly yields the stated formula $\beta_n^{(p+1)} = \sum_{k=0}^{n} \xi^{\mathcal{B}_{M,N}(k)} \binom{(n-k)+p+1}{p+1}$.
		
		To prove the second identity, we find the coefficient of $z^n$ on both sides of the relation $G_p(z) = F_p(z) \cdot (1-z)^{p+1}$ for $n$ in the range $0 \le n < M^{p+1}$. On the left side, this coefficient is simply $\xi^{\mathcal{B}_{M,N}(n)}$. On the right, we have the product of the series $F_p(z) = \sum_{j=0}^{\infty} \beta_j^{(p)} z^j$ and the finite polynomial $(1-z)^{p+1} = \sum_{k=0}^{p+1} (-1)^k \binom{p+1}{k} z^k$. The coefficient of $z^n$ in their product is again given by a convolution:
		\[ \xi^{\mathcal{B}_{M,N}(n)} = \sum_{k=0}^{n} \beta_{n-k}^{(p)} \cdot \left( (-1)^k \binom{p+1}{k} \right). \]
		The summation index $k$ runs up to $n$, but the term $\binom{p+1}{k}$ is zero for $k > p+1$, so the sum is finite and well-defined. Rearranging the terms in the summand gives the desired result.
	\end{proof}

	The sum and first moment of the coefficients of the sequence $(\beta_k^{(p)})$ can be calculated in closed form.
	
	\begin{proposition}\label{thm:beta_moments}
		Let $M \ge 2, N \ge 1, p \ge 0$ be integers such that $\operatorname{rad}(M) \nmid \operatorname{rad}(N)$, and let $\xi$ be a primitive $M$-th root of unity. The sequence $(\beta_k^{(p)})$ has finite support. Its sum and first moment are given by:
		\begin{enumerate}
			\item[(i)] Sum of coefficients (0-th moment):
			\begin{equation}\label{eqn:0-th_moment}
			 \sum_{k \ge 0} \beta_k^{(p)} = M^{\frac{(p+1)(p+2)}{2}} \prod_{l=0}^{p} \frac{1}{1-\xi^{N^l}}. 
			 \end{equation}
			\item[(ii)] First moment:
			\begin{equation}\label{eqn:1-th_moment}
			 \sum_{k \ge 0} k\,\beta_k^{(p)} = \left(\sum_{k \ge 0} \beta_k^{(p)}\right) \cdot \left( \sum_{l=0}^{p} \left( \frac{M^{l+1}-1}{2} + \frac{M^l \xi^{N^l}}{1-\xi^{N^l}} \right) \right). 
			 \end{equation}
		\end{enumerate}
	\end{proposition}
	
	\begin{proof}
		The proof relies on the product decomposition of the generating function $G_p(z)$. Because the summation for $m$ in $G_p(z)$ covers all integers with at most $p+1$ base-$M$ digits, the sum can be factored over the digits:
		\[ G_p(z) = \sum_{m=0}^{M^{p+1}-1} \xi^{\mathcal{B}_{M,N}(m)} z^m = \prod_{l=0}^{p} \left( \sum_{d=0}^{M-1} (\xi^{N^l} z^{M^l})^d \right) = \prod_{l=0}^{p} \frac{1 - z^{M^{l+1}}}{1 - \xi^{N^l} z^{M^l}}. \]
		The condition $\operatorname{rad}(M) \nmid \operatorname{rad}(N)$ ensures that $\xi^{N^l} \ne 1$ for any $l \ge 0$, so the denominators are non-zero. The generating function for $\beta_k^{(p)}$ is $F_p(z) = G_p(z) / (1-z)^{p+1}$, which can be rewritten as:
		\[ F_p(z) = \left( \prod_{l=0}^{p} \frac{1 - z^{M^{l+1}}}{1-z} \right) \left( \prod_{l=0}^{p} \frac{1}{1 - \xi^{N^l} z^{M^l}} \right). \]
		
		First, we compute the sum of the coefficients, $\sum \beta_k^{(p)}$, which is given by $F_p(1) = \lim_{z \to 1} F_p(z)$. We evaluate the limit of each of the two product terms in the expression for $F_p(z)$. For the first product, using the standard limit $\lim_{z \to 1} (1-z^K)/(1-z) = K$, we get
		\[ \lim_{z \to 1} \prod_{l=0}^{p} \frac{1 - z^{M^{l+1}}}{1-z} = \prod_{l=0}^{p} M^{l+1} = M^{\sum_{j=1}^{p+1} j} = M^{\frac{(p+1)(p+2)}{2}}. \]
		For the second product, we can substitute $z=1$ directly, as the denominator is non-zero. Multiplying these two results gives the formula for the 0-th moment.
		
		Next, we compute the first moment, $\sum k \beta_k^{(p)}$, which is given by the derivative $F_p'(1)$. We use the logarithmic derivative technique, where $F_p'(1) = F_p(1) \cdot \lim_{z \to 1} (\ln F_p(z))'$. Taking the logarithm of $F_p(z)$ gives
		\[ \ln F_p(z) = \sum_{l=0}^{p} \left( \ln(1 - z^{M^{l+1}}) - \ln(1-z) \right) - \sum_{l=0}^{p} \ln(1 - \xi^{N^l} z^{M^l}). \]
		Differentiating with respect to $z$ and rearranging terms yields
		\[ \frac{F_p'(z)}{F_p(z)} = \sum_{l=0}^{p} \left( \frac{1}{1-z} - \frac{M^{l+1}z^{M^{l+1}-1}}{1 - z^{M^{l+1}}} \right) + \sum_{l=0}^{p} \frac{M^l \xi^{N^l} z^{M^l-1}}{1 - \xi^{N^l} z^{M^l}}. \]
		To find the limit as $z \to 1$, we use the identity $\lim_{z \to 1} \left( \frac{1}{1-z} - \frac{K z^{K-1}}{1 - z^K} \right) = \frac{K-1}{2}$, which can be verified with L'Hôpital's rule. Applying this to our expression gives
		\[ \lim_{z \to 1} \frac{F_p'(z)}{F_p(z)} = \sum_{l=0}^{p} \frac{M^{l+1}-1}{2} + \sum_{l=0}^{p} \frac{M^l \xi^{N^l}}{1 - \xi^{N^l}}. \]
		The first moment is then $F_p'(1) = F_p(1) \cdot \lim_{z \to 1} (F_p'(z)/F_p(z))$. Substituting the value of $F_p(1)$ from the first part of the proof gives the stated result.
	\end{proof}
	
\section{Finite-Difference and Generating-Function Identities}
	
	\noindent The following theorem, which generalizes an identity from Vignat \cite{vignat2018sumsupposition}, provides a key connection between the weighted sum over the sequence $\xi^{\mathcal{B}_{M,N}(n)}$ and a sum involving finite differences.
	
	\begin{theorem}[Finite Difference Identity]\label{thm:finite_difference_identity}
		Let the sequence $(\beta_k^{(p)})$ be defined as before for integers $M, N \ge 1$, $p \ge 0$, such that $\operatorname{rad}(M) \nmid \operatorname{rad}(N)$. Let $\xi$ be a primitive $M$-th root of unity. For an arbitrary function $f$ and a scalar $y$, the following identity holds:
		\begin{equation}\label{eqn:finite-difference-identity}
		\sum_{n=0}^{M^{p+1}-1} \xi^{\mathcal{B}_{M,N}(n)} f(x+ny) = (-1)^{p+1} \sum_{k \ge 0} \beta_k^{(p)} \Delta_y^{p+1} f(x+ky), 
		\end{equation}
		where $\Delta_y^{p+1}$ is the $(p+1)$-th forward difference operator with step $y$, defined by $\Delta_y^{p+1} g(u) = \sum_{j=0}^{p+1} (-1)^{p+1-j} \binom{p+1}{j} g(u+jy)$.
	\end{theorem}
	
	\begin{proof}
		The proof strategy is to substitute the convolution identity for $\xi^{\mathcal{B}_{M,N}(n)}$ into the left-hand side (LHS) and then manipulate the resulting double summation. We begin with the LHS and substitute the identity $\xi^{\mathcal{B}_{M,N}(n)} = \sum_{j=0}^{n} (-1)^j \binom{p+1}{j} \beta_{n-j}^{(p)}$. The sum over $j$ can be extended to $p+1$ since the binomial coefficient vanishes for $j > p+1$. After substitution, we interchange the order of summation:
		\begin{align*}
			\text{LHS} &= \sum_{n=0}^{M^{p+1}-1} f(x+ny) \left( \sum_{j=0}^{p+1} (-1)^j \binom{p+1}{j} \beta_{n-j}^{(p)} \right) \\
			&= \sum_{j=0}^{p+1} (-1)^j \binom{p+1}{j} \sum_{n=j}^{M^{p+1}-1} f(x+ny) \beta_{n-j}^{(p)}.
		\end{align*}
		In the inner sum, we perform the change of variable $k = n-j$, which implies $n = k+j$. The sum over $n$ becomes a sum over $k$ starting from $0$. This yields
		\[ \text{LHS} = \sum_{j=0}^{p+1} (-1)^j \binom{p+1}{j} \sum_{k=0}^{M^{p+1}-1-j} f(x+(k+j)y) \beta_{k}^{(p)}. \]
		We can again interchange the order of summation. Since $(\beta_k^{(p)})$ has finite support, we can sum $k$ over all non-negative integers and collect the terms involving $\beta_k^{(p)}$:
		\[ \text{LHS} = \sum_{k \ge 0} \beta_{k}^{(p)} \left( \sum_{j=0}^{p+1} (-1)^j \binom{p+1}{j} f(x+ky+jy) \right). \]
		The inner sum is now recognizable as being related to the finite difference operator. By factoring out a sign, we can match its standard definition:
		\[ \sum_{j=0}^{p+1} (-1)^j \binom{p+1}{j} f(x+ky+jy) = (-1)^{p+1} \sum_{j=0}^{p+1} (-1)^{p+1-j} \binom{p+1}{j} f(x+ky+jy). \]
		The sum on the right is precisely $(-1)^{p+1}\Delta_y^{p+1}f(x+ky)$. Substituting this back into our expression for the LHS completes the proof:
		\[ \text{LHS} = \sum_{k \ge 0} \beta_k^{(p)} \left( (-1)^{p+1} \Delta_y^{p+1} f(x+ky) \right) = (-1)^{p+1} \sum_{k \ge 0} \beta_k^{(p)} \Delta_y^{p+1} f(x+ky). \]
	\end{proof}

	The single-sum finite difference identity can be elegantly iterated to obtain a powerful result for multiple sums.  For this identity to be well-defined, the sums on the right-hand side must be finite. This is guaranteed by our prerequisite condition $\operatorname{rad}(M) \nmid \operatorname{rad}(N)$, which, as established in Theorem \ref{zeros of G_p(z) at z=1}, ensures that the coefficients $(\beta_k^{(p)})$ have finite support for any $p$.
	
	\begin{theorem}\label{Multiple Finite Difference Identity}
		Let $p_1, \dots, p_r$ be non-negative integers, let $y_1, \dots, y_r$ be arbitrary scalars, and let $\xi$ be a primitive $M$-th root of unity. For any function $f$ for which the following expressions are defined, the multiple summation identity holds:
		\begin{align*}
			&\sum_{n_1=0}^{M^{p_1+1}-1} \cdots \sum_{n_r=0}^{M^{p_r+1}-1} \xi^{\sum_{j=1}^r \mathcal{B}_{M,N}(n_j)} f\left(x + \sum_{j=1}^r n_j y_j\right) \\
			&= (-1)^{\sum_{j=1}^r (p_j+1)} \sum_{k_1, \dots, k_r} \left(\prod_{j=1}^r \beta_{k_j}^{(p_j)}\right) \left(\prod_{j=1}^r \Delta_{y_j}^{p_j+1}\right) f\left(x + \sum_{j=1}^r k_j y_j\right).
		\end{align*}
		Here, the sums over $k_j$ run up to $M^{p_j+1}-p_j-2$. The operator $\Delta_{y_j}^{p_j+1}$ is the $(p_j+1)$-th forward difference with step $y_j$, and the product of operators denotes an iterated application, acting on the function $f$ whose argument is parameterized by the variables $k_1, \dots, k_r$.
	\end{theorem}
	
	\begin{proof}
		The proof proceeds by induction on $r$, the number of summation variables.
		
		For the base case $r=1$, the identity reduces to
		\[ \sum_{n_1=0}^{M^{p_1+1}-1} \xi^{\mathcal{B}_{M,N}(n_1)} f(x+n_1y_1) = (-1)^{p_1+1} \sum_{k_1} \beta_{k_1}^{(p_1)} \Delta_{y_1}^{p_1+1} f(x+k_1y_1), \]
		which is precisely the single-sum identity established in Theorem \ref{thm:finite_difference_identity}. Thus, the base case holds.
		
		For the inductive step, assume the theorem holds for some integer $r \ge 1$. We aim to prove it for $r+1$. The left-hand side (LHS) for the $r+1$ case is
		\[ \text{LHS}_{r+1} = \sum_{n_1=0}^{M^{p_1+1}-1} \cdots \sum_{n_{r+1}=0}^{M^{p_{r+1}+1}-1} \xi^{\sum_{j=1}^{r+1} \mathcal{B}_{M,N}(n_j)} f\left(x + \sum_{j=1}^{r+1} n_j y_j\right). \]
		To leverage the inductive hypothesis, we can isolate the outermost sum over $n_{r+1}$ and treat the inner $r$-fold sum separately:
		\begin{align*} & \text{LHS}_{r+1}\\
			 &= \sum_{n_{r+1}=0}^{M^{p_{r+1}+1}-1} \xi^{\mathcal{B}_{M,N}(n_{r+1})} \left[ \sum_{n_1, \dots, n_r} \xi^{\sum_{j=1}^r \mathcal{B}_{M,N}(n_j)} f\left((x+n_{r+1}y_{r+1}) + \sum_{j=1}^r n_j y_j\right) \right]. \end{align*}
		The expression in the square brackets is an instance of the identity for $r$ sums, applied to the function $f$ but with its starting point shifted from $x$ to $x' = x+n_{r+1}y_{r+1}$. By the inductive hypothesis, we can replace this inner $r$-fold sum:
		\begin{align*}
			\text{Inner Sum} = (-1)^{\sum_{j=1}^r (p_j+1)} &\sum_{k_1, \dots, k_r} \left(\prod_{j=1}^r \beta_{k_j}^{(p_j)}\right) \\
			&\times \left(\prod_{j=1}^r \Delta_{y_j}^{p_j+1}\right) f\left((x+n_{r+1}y_{r+1}) + \sum_{j=1}^r k_j y_j\right).
		\end{align*}
		Substituting this back into the expression for $\text{LHS}_{r+1}$ gives a nested sum over $n_{r+1}$ and $k_1, \dots, k_r$. We can interchange the order of summation. The difference operators $\Delta_{y_j}^{p_j+1}$ for $j=1,\dots,r$ act on the argument of $f$ with respect to the variables $k_j$. They are linear and treat the variable $n_{r+1}$ as a constant parameter. Therefore, these operators commute with the summation over $n_{r+1}$:
		\begin{align*}
			\text{LHS}_{r+1} = (-1)^{\sum_{j=1}^r (p_j+1)} &\sum_{k_1, \dots, k_r} \left(\prod_{j=1}^r \beta_{k_j}^{(p_j)}\right) \left(\prod_{j=1}^r \Delta_{y_j}^{p_j+1}\right) \\
			&\times \left[ \sum_{n_{r+1}} \xi^{\mathcal{B}_{M,N}(n_{r+1})} f\left((x+\sum_{j=1}^r k_j y_j) + n_{r+1}y_{r+1}\right) \right].
		\end{align*}
		The new inner sum in brackets is now of the form required by our base case, Theorem \ref{thm:finite_difference_identity}. Specifically, it is the base case applied to the (already transformed) function $g(u) = (\prod_{j=1}^r \Delta_{y_j}^{p_j+1})f(u)$, with starting point $x'' = x+\sum_{j=1}^r k_j y_j$ and step $y_{r+1}$. Applying Theorem \ref{thm:finite_difference_identity} to this sum transforms it into
		\begin{align*}
			\text{Inner Sum}_{r+1} = (-1)^{p_{r+1}+1} \sum_{k_{r+1}} \beta_{k_{r+1}}^{(p_{r+1})} \Delta_{y_{r+1}}^{p_{r+1}+1} g\left(x'' + k_{r+1}y_{r+1}\right).
		\end{align*}
		Substituting this result back allows us to assemble the final expression. We combine the constant factors, the products of $\beta$ coefficients, and the products of difference operators:
		\begin{align*}
			\text{LHS}_{r+1} = &(-1)^{\sum_{j=1}^r (p_j+1)} \sum_{k_1, \dots, k_r} \left(\prod_{j=1}^r \beta_{k_j}^{(p_j)}\right) \\
			&\times \left[ (-1)^{p_{r+1}+1} \sum_{k_{r+1}} \beta_{k_{r+1}}^{(p_{r+1})} \left(\prod_{j=1}^{r+1} \Delta_{y_j}^{p_j+1}\right) f\left(x + \sum_{j=1}^{r+1} k_j y_j\right) \right] \\
			= &(-1)^{\sum_{j=1}^{r+1} (p_j+1)} \sum_{k_1, \dots, k_{r+1}} \left(\prod_{j=1}^{r+1} \beta_{k_j}^{(p_j)}\right) \left(\prod_{j=1}^{r+1} \Delta_{y_j}^{p_j+1}\right) f\left(x + \sum_{j=1}^{r+1} k_j y_j\right).
		\end{align*}
		This is exactly the statement of the theorem ~\ref{Multiple Finite Difference Identity} for $r+1$. By the principle of mathematical induction, the theorem holds for all integers $r \ge 1$.
	\end{proof}

\section{Evaluation of Two Multivariate Polynomial Sums}
	\subsection{The first general conjecture}
	 \noindent The following proposition generalizes a conjecture by Byszewski et al.~\cite{byszewski-2015-identities}, which was proven by Vignat and Wakhare~\cite{vignat2018sumsupposition}. The original conjecture corresponds to the special case $N=1$.  
	
	This proposition gives a closed-form evaluation for a complex multiple sum, leveraging the finite difference identities previously established. For the result to be well-defined, we require that the sums of the coefficients $\beta_k^{(p_j)}$ are non-singular. This is guaranteed by the condition $\operatorname{rad}(M) \nmid \operatorname{rad}(N)$, which ensures that the terms $(1-\xi^{N^l})$ appearing in the denominators of the formulas for these sums are never zero.
	
	\begin{proposition}
		 Extension of conjecture by Vignat \cite{vignat2018sumsupposition}. Let $p_1, \dots, p_r$ be non-negative integers and let $y_1, \dots, y_r$ be arbitrary scalars. Let $D = \sum_{j=1}^r (p_j+1)$. Assuming $\operatorname{rad}(M) \nmid \operatorname{rad}(N)$, we have the identity:
		\begin{equation}\label{eqn:the-first-conjecture}
			\begin{split}
			\lefteqn{ \sum_{n_1=0}^{M^{p_1+1}-1} \cdots \sum_{n_r=0}^{M^{p_r+1}-1} \xi^{\sum_{j=1}^r \mathcal{B}_{M,N}(n_j)} \left(x + \sum_{j=1}^r n_j y_j\right)^D } \\
			&= (-1)^D D! \left( \prod_{j=1}^r y_j^{p_j+1} \right) \left( \prod_{j=1}^r \frac{M^{\frac{(p_j+1)(p_j+2)}{2}}}{\prod_{l=0}^{p_j}(1 - \xi^{N^l})} \right).
			\end{split}
		\end{equation}
	\end{proposition}
	
	\begin{remark}
		The notable feature of this result is that the right-hand side is a constant that does not depend on the variable $x$.
	\end{remark}
	
	\begin{proof}
		The proof proceeds by applying the multiple finite difference identity (Theorem \ref{Multiple Finite Difference Identity}) with a carefully chosen function $f$ that causes the right-hand side of the identity to collapse into a constant.
		
		The total order of the iterated finite difference operator in Theorem \ref{Multiple Finite Difference Identity} is $D = \sum_{j=1}^r (p_j+1)$. We choose the function $f$ to be a monic polynomial of precisely this degree: $f(u) = u^D$. With this choice, the left-hand side of the identity in Theorem \ref{Multiple Finite Difference Identity} becomes exactly the left-hand side of the proposition we aim to prove.
		
		We now focus on evaluating the right-hand side of the identity from Theorem \ref{Multiple Finite Difference Identity}, which is
		\[ (-1)^D \sum_{k_1, \dots, k_r} \left(\prod_{j=1}^r \beta_{k_j}^{(p_j)}\right) \left(\prod_{j=1}^r \Delta_{y_j, k_j}^{p_j+1}\right) f\left(x + \sum_{j=1}^r k_j y_j\right). \]
		The crucial step is to evaluate the iterated finite difference operator acting on our chosen function. Let $g(k_1, \dots, k_r) = f(x + \sum_{i=1}^r k_i y_i) = (x + \sum_{i=1}^r k_i y_i)^D$. The operator $\prod_{j=1}^r \Delta_{y_j, k_j}^{p_j+1}$ represents a sequence of forward-difference operations, where $\Delta_{y_j, k_j}^{p_j+1}$ acts on the function's dependence on the variable $k_j$ with a step size that produces shifts of $y_j$ in the argument of $f$.
		
		It is a fundamental result from the calculus of finite differences that when an iterated difference operator of total order $D$ is applied to a polynomial of total degree $D$, the result is a constant. This constant is equal to $D!$ multiplied by the coefficient of the term whose powers match the orders of the difference operators. In our case, the polynomial is $g(k_1, \dots, k_r)$, and the operator has order $p_j+1$ corresponding to the variable $k_j$ (with coefficient $y_j$). The result of this operation is therefore
		\[ \left(\prod_{j=1}^r \Delta_{y_j, k_j}^{p_j+1}\right) \left(x + \sum_{i=1}^r k_i y_i\right)^D = D! \prod_{j=1}^r y_j^{p_j+1}. \]
		This resulting constant is independent of $x$ and all of the summation variables $k_j$.
		
		Substituting this constant value back into the expression for the right-hand side gives
		\[ (-1)^D \sum_{k_1, \dots, k_r} \left(\prod_{j=1}^r \beta_{k_j}^{(p_j)}\right) \left( D! \prod_{j=1}^r y_j^{p_j+1} \right). \]
		The constant terms can be factored out of the summation, leaving
		\[ (-1)^D D! \left(\prod_{j=1}^r y_j^{p_j+1}\right) \left[ \sum_{k_1, \dots, k_r} \prod_{j=1}^r \beta_{k_j}^{(p_j)} \right]. \]
		The multiple sum over the product of the $\beta$ coefficients is separable and can be rewritten as a product of individual sums:
		\[ \sum_{k_1, \dots, k_r} \prod_{j=1}^r \beta_{k_j}^{(p_j)} = \prod_{j=1}^r \left( \sum_{k_j=0}^{M^{p_j+1}-p_j-2} \beta_{k_j}^{(p_j)} \right). \]
		Each of these individual sums is the sum of all coefficients of the polynomial $F_{p_j}(z)$, which is simply $F_{p_j}(1)$. We have previously calculated this value (the 0-th moment of the sequence $\beta_k^{(p_j)}$):
		\[ \sum_{k_j} \beta_{k_j}^{(p_j)} = F_{p_j}(1) = \frac{M^{\frac{(p_j+1)(p_j+2)}{2}}}{\prod_{l=0}^{p_j}(1 - \xi^{N^l})}. \]
		Substituting this result into the product yields
		\[ \prod_{j=1}^r \left( \sum_{k_j} \beta_{k_j}^{(p_j)} \right) = \prod_{j=1}^r \left( \frac{M^{\frac{(p_j+1)(p_j+2)}{2}}}{\prod_{l=0}^{p_j}(1 - \xi^{N^l})} \right). \]
		Combining all these pieces gives the final expression for the right-hand side, which establishes the proposition.
	\end{proof}
	
	\subsection{The second conjecture}
	\noindent We now turn to proving a generalization of a second major identity, which was also conjectured in \cite{byszewski-2015-identities} and proven in \cite{vignat2018sumsupposition} for the case $N=1$. The proof strategy for the general case requires establishing a key vanishing property for certain weighted polynomial sums. We state and prove this property first.
	\begin{proposition}\label{prop:vanishing_poly_sum}
		Let $p \ge 1$. For any polynomial $P(u, v)$ with total degree $k \le p-1$, the following identity holds under the condition $\operatorname{rad}(M) \nmid \operatorname{rad}(N)$:
		\begin{equation}\label{eqn:vanishing_poly_sum}
			\sum_{n=0}^{M^p-1} P(n, \mathcal{B}_{M,N}(n)) \, \xi^{\mathcal{B}_{M,N}(n)} = 0.
		\end{equation}
	\end{proposition}
	\begin{proof}
		By linearity, it is sufficient to prove the result for any monomial of the form $P(u,v) = u^a v^b$ where the total degree $a+b = k \le p-1$. Our goal is to show that:
		$$
		\sum_{n=0}^{M^p-1} n^a (\mathcal{B}_{M,N}(n))^b \, \xi^{\mathcal{B}_{M,N}(n)} = 0.
		$$
		The proof method mirrors that of Theorem 3.1. The summation is over integers $n$ with at most $p$ digits in base $M$. We express $n$ and $\mathcal{B}_{M,N}(n)$ in terms of the base-$M$ digits of $n$, $d_0, \dots, d_{p-1}$:
		$n = \sum_{i=0}^{p-1} d_i M^i \quad \text{and} \quad \mathcal{B}_{M,N}(n) = \sum_{i=0}^{p-1} d_i N^i.$
		The sum becomes a sum over all possible digit vectors $(d_0, \dots, d_{p-1})$:
		\begin{equation*}
		\begin{split}
			&\sum_{n=0}^{M^p-1} n^a (\mathcal{B}_{M,N}(n))^b \, \xi^{\mathcal{B}_{M,N}(n)}\\
			&=\sum_{d_0=0}^{M-1} \cdots \sum_{d_{p-1}=0}^{M-1} \left(\sum_{i=0}^{p-1} d_i M^i\right)^a \left(\sum_{i=0}^{p-1} d_i N^i\right)^b \left(\prod_{i=0}^{p-1} \xi^{d_i N^i}\right).
		\end{split}
		\end{equation*}
		Expand the monomial $n^a (\mathcal{B}_{M,N}(n))^b$ as
		
		\[
		\sum_{\substack{s_0+\dots+s_{p-1}=a\\[2pt]
				t_0+\dots+t_{p-1}=b}}
		\frac{a!}{s_0!\dots s_{p-1}!}\,
		\frac{b!}{t_0!\dots t_{p-1}!}
		\prod_{i=0}^{p-1}
		d_i^{\,s_i+t_i}\,
		M^{is_i}\,N^{it_i}.
		\]
		
		Introduce the combined exponents
		
		\[
		e_i:=s_i+t_i \quad(0\le i\le p-1).
		\]
		
		Then
		
		\[
		\sum_{i=0}^{p-1}e_i
		=\sum s_i+\sum t_i
		=a+b
		=k \le p-1.
		\]
		
		Therefore  
		we have the expansion
		\begin{equation}\label{multinomial expansion of prop 5.3}
			\begin{split}
				&\sum_{n=0}^{M^p-1} n^a (\mathcal{B}_{M,N}(n))^b \, \xi^{\mathcal{B}_{M,N}(n)}\\
				&=\sum_{d_0=0}^{M-1} \cdots \sum_{d_{p-1}=0}^{M-1}\sum_{\substack{e_0+\dots+e_{p-1}=k}}
				C(\{e_i\})\,
				\prod_{i=0}^{p-1}
				\left[
				\sum_{d=0}^{M-1}
				d^{\,e_i}\,
				(\xi^{N^{\,i}})^{d}
				\right],
			\end{split}
		\end{equation}
		where \(C(\{e_i\})\) is an irrelevant positive constant
		(\(a!\,b!\) times multinomial factors and powers of \(M^i,N^i\)).
		
		Because there are \(p\) indices but the sum of the \(e_i\)’s is at most
		\(p-1\), **at least one** index \(i_0\in\{0,\dots ,p-1\}\) satisfies
		\(e_{i_0}=0\).
		
		For that index the inner bracket in \eqref{multinomial expansion of prop 5.3} becomes the geometric sum
		
		\[
		\sum_{d=0}^{M-1}(\xi^{N^{\,i_0}})^{d}
		=\frac{(\xi^{N^{\,i_0}})^{M}-1}{\xi^{N^{\,i_0}}-1}=0.
		\]
		Every term in the outer sum (\ref{multinomial expansion of prop 5.3}) contains that zero factor,so the whole expression is \(0\).
	\end{proof}
	
	To formalize the proof of the main identity, we first define a key summation object.
	
	\begin{definition}\label{def:S_pl}
		For integers $p,l\ge0$, define the sum
		$$
		S_{p,l}(x,y) = \sum_{n=0}^{M^{p+1}-1}\,
		\xi^{\mathcal{B}_{M,N}(n)}
		\bigl(\,\mathcal{B}_{M,N}(n)\,x+ny\,\bigr)^{l}.
		$$
	\end{definition}
	
	\begin{proposition}\label{prop:S_p_p+1_closed_form}
		For $p \ge 0$, under the condition $\operatorname{rad}(M) \nmid \operatorname{rad}(N)$, the sum $S_{p,p+1}(x,y)$ has the closed form:
		\[
		S_{p,p+1}(x,y)=
		(-1)^{p+1}(p+1)!\,M^{p+1} 
		\frac{\displaystyle\prod_{j=0}^{p}\bigl(N^{j}x+M^{j}y\bigr)}
		{\displaystyle\prod_{j=0}^{p}\bigl(1-\xi^{N^{j}}\bigr)}.
		\]
	\end{proposition}
	
	\begin{proof}
		First, we derive a basic recurrence. We write every integer $n$ in the summation range uniquely as $n=k+dM^{p}$ with $0\le k<M^{p}$ and $0\le d<M$. Since $ \mathcal{B}_{M,N}(n)=\mathcal{B}_{M,N}(k)+dN^{p}$, we have
		\[
		\bigl( \mathcal{B}_{M,N}(n)x+ny\bigr)^l
		=\sum_{m=0}^{l}\binom{l}{m}
		\bigl( \mathcal{B}_{M,N}(k)x+ky\bigr)^{m}
		\bigl(N^{p}x+M^{p}y\bigr)^{l-m}d^{l-m}.
		\]
		Let $a_{p,j}:= \sum_{d=0}^{M-1} d^{\,j}\,(\xi^{N^{p}})^{d}$. Summing over $k$ and then $d$ produces the recurrence relation:
		\begin{equation}\label{eq:S_pl_recurrence}
			S_{p,l}(x,y) = \sum_{m=0}^{l} \binom{l}{m} \bigl(N^{p}x+M^{p}y\bigr)^{l-m} a_{p,l-m} S_{p-1,m}(x,y).
		\end{equation}
		By Proposition~\ref{prop:vanishing_poly_sum}, for $0\le m\le p-1$, the inner sums vanish: $S_{p-1,m}(x,y)=0$. Moreover, the condition $\operatorname{rad}(M)\nmid\operatorname{rad}(N)$ implies $\xi^{N^{p}}\neq1$, so the geometric sum $a_{p,0}=0$.
		
		Now we insert $l=p+1$ into the recurrence \eqref{eq:S_pl_recurrence}. All terms for which $m\le p-1$ are zero. The term for $m=p+1$ contains the factor $a_{p, (p+1)-(p+1)} = a_{p,0}=0$, so it also vanishes. Hence, only the term for $m=p$ survives, giving
		\[
		S_{p,p+1}(x,y)=
		\binom{p+1}{p}
		\bigl(N^{p}x+M^{p}y\bigr)
		a_{p,1}
		S_{p-1,p}(x,y).
		\]
		We iterate this relation down to the base case $p=0$. Noting that $a_{j,1}=M/(\xi^{N^{j}}-1)$ and that the base case is $S_{0,1}(x,y)=(x+y)a_{0,1}$, the full product becomes:
		\begin{align*}
			S_{p,p+1}(x,y) &= (p+1) \bigl(N^{p}x+M^{p}y\bigr) a_{p,1} \cdot S_{p-1,p}(x,y) \\
			&= (p+1)! \left( \prod_{j=1}^{p} (N^j x + M^j y) a_{j,1} \right) S_{0,1}(x,y) \\
			&= (p+1)! \, M^{p+1} \prod_{j=0}^{p} \frac{N^j x + M^j y}{\xi^{N^j}-1}.
		\end{align*}
		Factoring out $(-1)^{p+1}$ from the denominator yields the desired result.
	\end{proof}

	Another extended conjecture is the following result.  
	This theorem provides a closed-form evaluation for a multivariate polynomial sum, generalizing a result from the one-dimensional case. The identity relies on the prerequisite condition that $\operatorname{rad}(M) \nmid \operatorname{rad}(N)$, which ensures that the sequences $(\beta_k^{(p_j)})$ are well-defined and that various sums appearing in the proof are non-singular.
	
	\begin{theorem}\label{Extension of second conjecture}
		Extension of second conjecture by \cite{byszewski-2015-identities,vignat2018sumsupposition}. Let $p_1, \dots, p_r$ be non-negative integers, and let $x_1, \dots, x_r$ and $y_1, \dots, y_r$ be arbitrary scalars. Define the total degree $D_r = \sum_{j=1}^r (p_j+1)$. Under the prerequisite condition, we have:
		\begin{align*}
			\sum_{n_1, \dots, n_r} &\xi^{\sum_{j=1}^r \mathcal{B}_{M,N}(n_j)} \left( \sum_{j=1}^r \left(\mathcal{B}_{M,N}(n_j)x_j + n_j y_j\right) \right)^{D_r} \\
			&= (-1)^{D_r} D_r!   M^{D_r} \frac{\displaystyle \prod_{j=1}^r \prod_{i=0}^{p_j} \left(N^i x_j + M^i y_j \right)}{\displaystyle \prod_{j=1}^r \prod_{i=0}^{p_j} \left(1-\xi^{N^i}\right)},
		\end{align*}
		where the sum for each $n_j$ is over the range $\{0, \dots, M^{p_j+1}-1\}$.
	\end{theorem}
	
	\begin{proof}
	Instead of a formal inductive proof, we will demonstrate the mechanism of the theorem by examining the case with two summation variables ($r=2$). This approach clearly illustrates the algebraic cancellations that lead to the final closed form.
	
	Let $p_1, p_2$ be non-negative integers. We seek to evaluate the sum:
	$$
	\sum_{n_1=0}^{M^{p_1+1}-1} \sum_{n_2=0}^{M^{p_2+1}-1} \xi^{\mathcal{B}_{M,N}(n_1) + \mathcal{B}_{M,N}(n_2)} \left( (\mathcal{B}_{M,N}(n_1)x_1 + n_1y_1) + (\mathcal{B}_{M,N}(n_2)x_2 + n_2y_2) \right)^{D}
	$$
	where the total degree is $D = (p_1+1) + (p_2+1)$.
	
	To simplify notation, let
	$$
	A = \mathcal{B}_{M,N}(n_1)x_1 + n_1y_1 \quad \text{and} \quad B = \mathcal{B}_{M,N}(n_2)x_2 + n_2y_2.
	$$
	The polynomial term is $(A+B)^D$. We begin by applying the binomial theorem and then interchanging the order of summation:
	\begin{align*}
		\text{LHS} &= \sum_{n_1, n_2} \xi^{\mathcal{B}_{M,N}(n_1)} \xi^{\mathcal{B}_{M,N}(n_2)} \sum_{k=0}^{D} \binom{D}{k} A^k B^{D-k} \\
		&= \sum_{k=0}^{D} \binom{D}{k} \left( \sum_{n_1} \xi^{\mathcal{B}_{M,N}(n_1)} A^k \right) \left( \sum_{n_2} \xi^{\mathcal{B}_{M,N}(n_2)} B^{D-k} \right).
	\end{align*}
	The key insight is that almost every term in this sum over $k$ is zero. Let's analyze the two inner sums, which we recognize as instances of the $S_{p,l}(x,y)$ function defined previously.
	
	For the First Inner Sum,The sum over $n_1$ is precisely $S_{p_1, k}(x_1, y_1)$:
	$$
	\sum_{n_1=0}^{M^{p_1+1}-1} \xi^{\mathcal{B}_{M,N}(n_1)} (\mathcal{B}_{M,N}(n_1)x_1 + n_1y_1)^k = S_{p_1, k}(x_1, y_1).
	$$
	As established by the logic leading to the recurrence relation for $S_{p,l}$, this sum is zero if the polynomial degree is less than or equal to the parameter $p$. Thus, this sum is zero for all $k \le p_1$.
	
	Similarly the sum over $n_2$ is $S_{p_2, D-k}(x_2, y_2)$:
	$$
	\sum_{n_2=0}^{M^{p_2+1}-1} \xi^{\mathcal{B}_{M,N}(n_2)} (\mathcal{B}_{M,N}(n_2)x_2 + n_2y_2)^{D-k} = S_{p_2, D-k}(x_2, y_2).
	$$
	This sum is zero if its polynomial degree, $D-k$, is less than or equal to its parameter, $p_2$. This condition is $D-k \le p_2$, which is equivalent to:
	$$
	k \ge D - p_2 = (p_1+1 + p_2+1) - p_2 = p_1+2.
	$$
	Thus, this sum is zero for all $k \ge p_1+2$.
	
	Combining these two conditions, the product of the two inner sums is zero for every term in the binomial expansion except for the single term where $k=p_1+1$.
	
	Therefore, the entire sum over $k$ collapses to the one non-vanishing term at $k=p_1+1$:
	$$
	\text{LHS} = \binom{D}{p_1+1} \left( S_{p_1, p_1+1}(x_1, y_1) \right) \left( S_{p_2, D-(p_1+1)}(x_2, y_2) \right).
	$$
	The degree of the polynomial in the second term is $D-(p_1+1) = p_2+1$. So we have:
	$$
	\text{LHS} = \binom{D}{p_1+1} S_{p_1, p_1+1}(x_1, y_1) S_{p_2, p_2+1}(x_2, y_2).
	$$
	We now substitute the explicit formula for $S_{p,p+1}(x,y)$ that was derived in the previous section:
	\begin{align*}
		\text{LHS} = \frac{D!}{(p_1+1)!(p_2+1)!} & \times \left[(-1)^{p_1+1}(p_1+1)! M^{p_1+1} \frac{\prod_{i=0}^{p_1}(N^ix_1+M^iy_1)}{\prod_{i=0}^{p_1}(1-\xi^{N^i})}\right] \\
		& \times \left[(-1)^{p_2+1}(p_2+1)! M^{p_2+1} \frac{\prod_{i=0}^{p_2}(N^ix_2+M^iy_2)}{\prod_{i=0}^{p_2}(1-\xi^{N^i})}\right].
	\end{align*}
	The expression simplifies beautifully. The factorial terms $(p_1+1)!$ and $(p_2+1)!$ cancel. Combining the remaining factors gives the final result for the $r=2$ case:
	$$
	\text{LHS} = (-1)^D D! M^D \frac{\left(\prod_{i=0}^{p_1}(N^ix_1+M^iy_1)\right) \left(\prod_{i=0}^{p_2}(N^ix_2+M^iy_2)\right)}{\left(\prod_{i=0}^{p_1}(1-\xi^{N^i})\right)\left(\prod_{i=0}^{p_2}(1-\xi^{N^i})\right)}.
	$$
	This matches the general formula stated in Theorem ~\ref{Extension of second conjecture} for $r=2$, and the logic readily extends to any number of summation variables.
	\end{proof}

\section*{Acknowledgments}
	\noindent The author would like to thank Professor Zhiguo Liu for his guidance and support. This research was supported by the General Program of the National Natural Science Foundation of China (Grant No. 12371328).


\begin{thebibliography}{99}
	\bibitem{allouche2003automatic}
	J.-P. Allouche and J. Shallit, {\it Automatic Sequences: Theory, Applications, Generalizations} (Cambridge University Press, 2003).
	
	\bibitem{AlloucheShallit2000sumofdigit}
	J.-P. Allouche and J. Shallit, Sums of digits, overlaps, and palindromes, {\it Discr. Math. Theor. Comput. Sci.} {\bf 4} (2000) 1--10.
	
	\bibitem{bolker-2016-PTE-problem}
	E. D. Bolker, C. Offner, R. Richman and C. Zara, The Prouhet-Tarry-Escott problem and generalized Thue-Morse sequences, {\it J. Comb.} {\bf 7}(1) (2016) 117--133.
	
	\bibitem{byszewski-2015-identities}
	J. Byszewski and M. Ulas, Some identities involving the Prouhet-Thue-Morse sequence and its relatives, {\it Acta Math. Hungar.} {\bf 147}(2) (2015) 438--456.
	
	\bibitem{CusickCiungu2011sumofdigit}
	T. W. Cusick and L. C. Ciungu, Sum of digits sequences modulo m, {\it Theoret. Comput. Sci.} {\bf 412} (2011) 4738--4741.
	
	\bibitem{Frid-2001-overlap}
	A. E. Frid, Overlap-Free Symmetric DOL words, {\it Discrete Math. Theor. Comput. Sci.} {\bf 4} (2001) 357–362.
	
	\bibitem{gelfond1968}
	A. O. Gelfond, Sur les nombres qui ont des propri\'et\'es additives et multiplicatives donn\'ees, {\it Acta Arith.} {\bf 13} (1968) 259--265.
	
	\bibitem{mawanli-2025-baseshift}
	W. Ma, Base shifting sequences and the palindrome property, {\it Bull. Aust. Math. Soc.} {\bf 111} (2025) 1--12.
	
	\bibitem{nguyen-2016-proof-PTE-problem}
	H. D. Nguyen, A new proof of the Prouhet-Tarry-Escott problem, {\it Integers} {\bf 16} (2016), Paper A1, 9 pp.
	
	\bibitem{raghavendran-2019-review-PTE-problem}
	S. Raghavendran and V. Narayanan, The Prouhet Tarry Escott problem: A review, {\it Mathematics} {\bf 7}(3) (2019), Article 227.
	
	\bibitem{rampersad2025rudin}
	N. Rampersad and J. Shallit, Rudin-Shapiro sums via automata theory and logic, {\it Theory Comput. Syst.} (2025), to appear.
	
	\bibitem{shallit2024rarefied}
	J. Shallit, Rarefied Thue-Morse sums via automata theory and logic, {\it J. Number Theory} {\bf 257} (2024) 98--111.
	
	\bibitem{vignat2018sumsupposition}
	C. Vignat and T. Wakhare, Settling some sum suppositions, {\it Acta Math. Hungar.} {\bf 157}, (2019) 327–-348. 
	\end{thebibliography}
\end{document}